\documentclass[11pt]{amsart}

\usepackage{amsmath,amsfonts,amsthm,graphicx}

\newcommand{\op}{\operatorname}
\newcommand{\tb}{\textit{tb}}
\newcommand{\<}{\langle}
\renewcommand{\>}{\rangle}
\newcommand{\Z}{\mathbb{Z}}
\newcommand{\R}{\mathbb{R}}
\newcommand{\A}{\mathcal{A}}
\newcommand{\st}{J^1(S^1)}
\def\fig#1{\raisebox{-2.2ex}{\includegraphics[height=5.2ex]{#1}}}
\def\figg#1{\raisebox{-2.1ex}{\includegraphics[width=5.2ex]{#1}}}

\def\figgg#1{\raisebox{-.3ex}{\includegraphics[width=5.2ex]{#1}}}

\theoremstyle{plain}
\newtheorem{theorem}{Theorem}[section]
\newtheorem{corollary}[theorem]{Corollary}
\newtheorem{lemma}[theorem]{Lemma}

\newtheorem{proposition}[theorem]{Proposition}

\theoremstyle{definition}
\newtheorem{definition}[theorem]{Definition}

\newtheorem{remark}[theorem]{Remark}

\theoremstyle{remark}
\newtheorem*{example}{Example}

\numberwithin{equation}{section}

\title[Invariants of Legendrian solid torus links]{Generalized normal rulings and invariants of Legendrian solid torus links}
\author{Mikhail Lavrov}
\address{Carnegie Mellon University, Pittsburgh, PA 15213}
\email{mlavrov@andrew.cmu.edu}

\author{Dan Rutherford}
\address{University of Arkansas, Fayetteville, AR 72701}
\email{drruther@uark.edu}

\begin{document}

\begin{abstract}  For Legendrian links in the $1$-jet space of $S^1$ we show that the $1$-graded ruling polynomial may be recovered from the Kauffman skein module.  For such links a generalization of the notion of normal ruling is introduced.  We show that the existence of such a generalized normal ruling is equivalent to sharpness of the Kauffman polynomial estimate for the Thurston-Bennequin number as well as to the existence of an ungraded augmentation of the Chekanov-Eliashberg DGA.  Parallel results involving the HOMFLY-PT polynomial and $2$-graded generalized normal rulings are established.
\end{abstract}

\maketitle

\section{Introduction}

In $\R^3$ interesting connections exist between the $2$-variable knot polynomials and invariants of Legendrian knots.  With respect to the standard contact structure on $\R^3$, Fuchs and Tabachnikov \cite{FT} showed that an upper bound for the Thurston-Bennequin number arises from the Kauffman and HOMFLY-PT knot polynomials.  Furthermore, when this estimate is sharp some non-classical invariants exhibit nice properties.  Specifically, combining results from \cite{F}, \cite{FI}, \cite{S}, and \cite{R1} we have:

\begin{theorem} \label{thm:Review}
For a Legendrian link $L \subset \R^3$ the following three statements are all equivalent:

\begin{enumerate}

\item The estimate $\tb(L) \leq -\deg_a F_L$ (resp. $\tb(L) \leq -\deg_a P_L$) is sharp where $F_L, P_L \in \Z[a^{\pm 1}, z^{\pm 1}]$ denote the Kauffman and HOMFLY-PT polynomials.  

\item A front diagram for $L$ has a $1$-graded (resp. $2$-graded) normal ruling.

\item The Chekanov-Eliashberg DGA of $L$ has a $1$-graded (resp. $2$-graded) augmentation.

\end{enumerate}
\end{theorem}
\noindent In addition, the conditions of Theorem \ref{thm:Review} are necessary in order for $L$ to have a linear at infinity generating family.  

In this article, we establish analogous results for Legendrian knots in the $1$-jet space of the circle, $J^1(S^1)$.  The manifold $J^1(S^1)$ is topologically an open solid torus and carries a standard contact structure.  Legendrian knots in $J^1(S^1)$ have attracted a fair amount of attention in the literature; see for instance \cite{DG}, \cite{NgTr}, \cite{T1}.  The $1$-jet space setting comes with convenient projections from which Legendrian knots may be presented via front or Lagrangian diagrams and Legendrian isotopy may be described in a combinatorial manner.  In addition, $1$-jet spaces provide a natural setting for the use of generating families.

A convenient formal way to define a normal ruling, $\rho$, of $L$ is as a family of fixed point free involutions of the strands of the front diagram of $L$ subject to many restrictions.  At least locally, this may be viewed as a decomposition of the front diagram into pairs of paths.  Chekanov and Pushkar introduced normal rulings in \cite{ChP}--albeit with different terminology--as well as related Legendrian isotopy invariants  which have become known as ruling polynomials.  
In connection with augmentations, Fuchs independently defined normal rulings of knots in $\R^3$ and, in the case of the Kauffman polynomial, already conjectured the equivalence of (1) and (2) in \cite{F}.  This conjecture was verified in \cite{R1} where it was shown that in fact the $1$-graded and $2$-graded ruling polynomials appear as coefficients of the Kauffman and HOMFLY-PT polynomials respectively.  

Relationships between the Kauffman/HOMFLY-PT invariants and Legendrians knots in $J^1(S^1)$ have already begun to be studied, and several factors make the situation more interesting.  For instance, the HOMFLY-PT polynomial, $P_L$, of a solid torus link, $L$, belongs to a polynomial algebra over $R = \Z[a^{\pm 1}, z^{\pm 1}]$ with a countably infinite number of generators $A_k, k \in \Z\setminus \{0\}$; the Kauffman polynomial has a similar form.  Chmutov and Goryunov \cite{CG} proved Thurston-Bennequin number estimates analogous to those appearing in (1) of Theorem \ref{thm:Review} using these many variable Kauffman and HOMFLY-PT polynomials.  In the case of the HOMFLY-PT polynomial, it was shown in \cite{R2} that the $2$-graded ruling polynomial can be recovered from the HOMFLY-PT polynomial,  but this requires first specializing via an $R$-module homomorphism $R[A_{\pm 1}, A_{\pm 2}, \ldots] \rightarrow R$.  In the present work we develop analogous results involving the $1$-graded ruling polynomial and the Kauffman skein module.  (See Theorems \ref{thm:front} and \ref{thm:spec}.)  

The need to specialize the Kauffman and HOMFLY-PT invariants in order to recover the ruling polynomials has an interesting consequence.  There are many solid torus links where the Kauffman or HOMFLY-PT polynomial estimate is sharp, yet the corresponding ruling polynomial vanishes.  As a result, for Legendrians in $J^1(S^1)$ some adjustment is required to statement (2) of Theorem \ref{thm:Review}.  For this purpose, we introduce a quite natural notion of {\it generalized normal ruling} where the fixed point free condition is relaxed.  Our main result is the following analog of Theorem \ref{thm:Review}:

\begin{theorem} \label{the:Main}
 Let $L \subset \st$ be a Legendrian link.   

\begin{enumerate}

\item Then, the estimate $\tb(L) \leq -\deg_a F_L$ (resp. $\tb(L) \leq -\deg_a P_L$) is sharp if and only if $L$ has a $1$-graded (resp. $2$-graded) generalized normal ruling.  

\item Suppose $L$ has been assigned a $\Z/p$-valued Maslov potential. Then, the Chekanov-Eliashberg DGA of $L$ has a $p$-graded augmentation if and only if a front diagram for $L$ admits a $p$-graded generalized normal ruling.  
\end{enumerate}
\end{theorem}

\begin{remark}
Aside from allowing the more general $p$-graded condition in (2), it is natural to organize the three statements into these two equivalences.  Even in $\R^3$, the authors do not know of any proof of an implication between the statements about the knot polynomial estimates and existence of augmentations which is able to avoid using normal rulings.  There are settings, for instance certain contact lens spaces, where Legendrian contact homology \cite{Li} and  HOMFLY-PT polynomial estimates (\cite{C1}, \cite{C2}) for $\tb$ have been established while an appropriate notion of normal ruling has yet to be formulated.  For this reason, establishing a more direct link between Bennequin type inequalities and augmentations could prove interesting.  
\end{remark}

\subsection{Organization}

The article is arranged as follows:  In Section 2, we provide the necessary background about normal rulings and the Kauffman and HOMFLY-PT invariants and also introduce generalized normal rulings.  Section 3  runs parallel to the results on the HOMFLY-PT skein module and $2$-graded rulings from \cite{R2}.  We show how to recover the $1$-graded ruling polynomial from an appropriate specialization of the Kauffman skein module.  A natural basis for the Kauffman skein module is indexed by partitions, and for this basis we provide an explicit formula for the specialization.  In Section 4 we prove (1) of Theorem \ref{the:Main} by combining the results of Section 3 (and \cite{R2} for the HOMFLY-PT case) with a linear independence argument.  

The final section of the article deals with part (2) of Theorem \ref{the:Main}.  For the forward implication we base all of our arguments on the linear algebraic results of Barannikov \cite{B} from which the reason behind the normality conditions, with or without fixed points, becomes clear.  

\subsection{Acknowledgements}  This work was initiated through the PRUV program at Duke University.  We thank David Kraines for supervising the program and encouraging our participation.  Also, we thank Lenny Ng for his interest in the project.  
A portion of the writing was carried out while the second author was a visitor at the Max Planck Institute for Mathematics in Bonn, and it is a pleasure to acknowledge MPIM for their hospitality.  The first author received support from NSF CAREER grant DMS-0846346.  
    
\section{Background on Legendrian solid torus links}

We assume familiarity with basic concepts about Legendrian knots such as {\it front projections}, {\it Legendrian Reidemeister moves}, {\it Thurston Bennequin number}, and {\it rotation number} at least for knots in $\R^3$.  See, for instance, \cite{G}, and also note that \cite{R2} contains an alternate discussion of the case of Legendrian knots in $J^1(S^1)$.

We view the $1$-jet space of the circle, $J^1(S^1)$, as $S^1 \times \R^2$ equipped with the contact structure $\xi = \ker(dz - y\, dx)$ where $x$ is a circle-valued coordinate.  We occasionally refer to a (Legendrian) link $L \subset J^1(S^1)$ as a {\it (Legendrian) solid torus link}.  The front projection of a Legendrian solid torus link consists of some number of closed curves in the $xz$-annulus which we view as $[0,1]\times \R$ with the identification $(0, z)\sim (1,z)$.  Generically, front projections are immersed and embedded except at semi-cubical cusps and transverse double points, and two such projections represent Legendrian isotopic links if and only if they are related by a sequence of Legendrian Reidemeister moves. 

We make the convention of extending the Thurston-Bennequin number to homologically non-trivial links by using the front projection formula
\[
\tb(L) = w(L) - c(L)
\]
where $w(L)$ denotes the writhe of $L$ (a signed sum of crossings) and $c(L)$ is half the number of cusps of $L$. 

Similarly, for a Legendrian knot $L \subset J^1(S^1)$ we will define the rotation number as
\[
r(L) = \frac{1}{2}(d(L) - u(L) )
\]
where $d(L)$ denotes the number of downward oriented cusps and $u(L)$ the number of upward oriented cusps.  

\subsection{Products of basic fronts}

Given two annular front diagrams, $K$ and $L$, we define the product, $K\cdot L$, by stacking $K$ above $L$.  In contrast to the case of smooth knot diagrams, this product is non-commutative as the Legendrian isotopy types of $K\cdot L$ and $L \cdot K$ will not agree in general; see \cite{T1} and \cite{R2}.

In this article the {\it basic fronts}, $A_m$, will play an important role.  Given $m \in \Z_{>0}$, $A_m$ is the front diagram that winds $m$ times around the annulus with $m-1$ crossings and no cusps; see Figure \ref{fig:BFront}.  When it is necessary to pay attention to orientations, for $m>0$, we will use $A_m$ (resp. $A_{-m}$) for the basic front oriented in the direction of the positive (resp. negative) $x$-axis.

\begin{figure} 
\centerline{\includegraphics{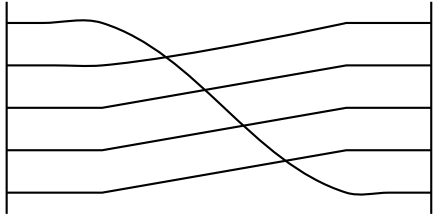}}
\caption{The basic front $A_5$.}
\label{fig:BFront}
\end{figure}

Given an $\ell$-tuple of positive integers $\lambda = (\lambda_1, \ldots, \lambda_\ell)$ we write $A_\lambda = A_{\lambda_1} A_{\lambda_2} \cdots A_{\lambda_\ell}$ for the product of basic fronts and $A_{-\lambda}$ for the product with all orientations reversed.

\subsection{Kauffman polynomial in $J^1(S^1)$}  \label{sec:KP} We now describe a generalization from \cite{Tu} of the Kauffman polynomial to smooth links (not necessarily Legendrian) in the solid torus.  In practice, this invariant is computed by reducing a link diagram to products of basic fronts via skein relations.  Whenever appropriate, we will view a front diagram of a Legendrian link as a smooth link diagram by placing the strand with lesser slope on top at crossings and smoothing cusps.  

Let $\mathcal{D}$ denote the set of regular isotopy classes of unoriented link diagrams in the annulus.  That is, we consider link diagrams up to the equivalence generated by Type II and Type III Reidemeister moves.   Using the coefficient ring $R = \Z[a^{\pm 1}, z^{\pm 1}]$ we define the Kauffman skein module $\mathcal{F}$ as the quotient of  the free $R$-module,  $R\mathcal{D}$, by the sub-module generated by the Kauffman skein relations,
\begin{equation} \label{eq:SR1}
\figg{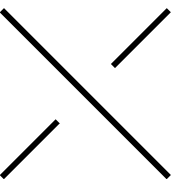} - \figg{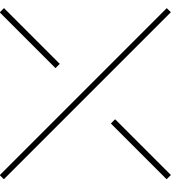} = z \left(\, \figg{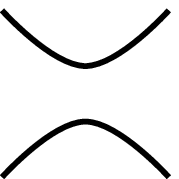} - \figg{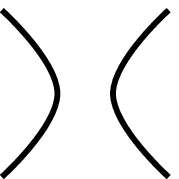} \,\right),
\end{equation}  
\begin{equation} \label{eq:SR2}
\figg{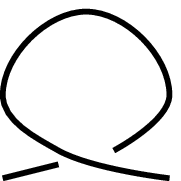} = a \left( \figgg{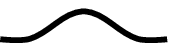} \right), \quad \figg{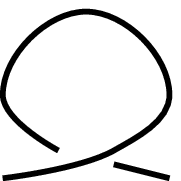} = a^{-1} \left( \figgg{KSR7.eps} \right), \mbox{and}
\end{equation}  
\begin{equation} \label{eq:SR3}
\figg{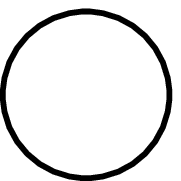} \bigsqcup L = \left( \frac{a-a^{-1}}{z} +1 \right) \cdot L. 
\end{equation}

The product of diagrams gives a well defined product on $\mathcal{F}$ which is commutative as we now consider diagrams of smooth links rather than front diagrams of Legendrian links.  Turaev showed in \cite{Tu} that $\mathcal{F}$ is a polynomial $R$-algebra in the basic fronts.  Thus, to a link diagram $L$ we may associate a polynomial $D_L(a,z; A_1, A_2,\ldots)$ according to
\[
\begin{array}{ccc} \mathcal{F} & \cong & R[A_1, A_2, \ldots] 
\\

									[L] & \leftrightarrow & D_L  
\end{array}.
\]
The {\it Kauffman polynomial} of an oriented link $L \subset J^1(S^1)$ is then defined by the normalization $F_L = a^{-w(L)} D_L$ where $w(L)$ denotes the writhe of $L$.

Chmutov and Goryunov \cite{CG} proved that for any Legendrian link $L \subset J^1(S^1)$,
\begin{equation}\label{eq:CGEst}
\tb(L) \leq -\deg_a F_L.
\end{equation}
While \cite{CG} uses a different projection annulus for computing $F_L$, a proof of (\ref{eq:CGEst}) matching our conventions for $F_L$ may be given precisely as in the case of the HOMFLY-PT polynomial addressed in Section 6.2 of \cite{R2}.

\begin{remark} \label{rem:basis}
(i) Recall that a possibly empty sequence of positive integers, $\lambda = (\lambda_1, \ldots, \lambda_\ell)$, is called a {\it partition} if $\lambda_1 \geq \ldots \geq \lambda_\ell$.  The integers $\lambda_i$ are called the {\it parts} of $\lambda$ and we sometimes use the notation $\lambda = 1^{j_1} 2^{j_2} \cdots n^{j_n}$ to indicate that $\lambda$ is the partition with $j_r$ parts equal to $r$, $r= 1, \ldots, n$.  As it will be useful later, we note that the collection of products $A_\lambda$ with $\lambda$ a partition forms an $R$-module basis for $\mathcal{F}$.  

(ii) The HOMFLY-PT skein module is defined in a similar manner using oriented link diagrams and an appropriate modification of the skein relations (\ref{eq:SR1})-(\ref{eq:SR3}) (see, for instance \cite{R2}).  The result is a polynomial algebra generated by the oriented basic fronts \cite{Tu}.  For a given oriented link $L \subset J^1(S^1)$ we denote the corresponding HOMFLY-PT polynomial as 
\[
P_L \in R[A_{\pm 1}, A_{\pm 2}, \ldots].
\]
\end{remark}

\subsection{Normal rulings in $J^1(S^1)$}
Let $L \subset S^1 \times \R$ be the front projection of a Legendrian link in the solid torus satisfying the additional assumption that all crossings and cusps have distinct $x$-coordinates none of which equals $0$.  A normal ruling can be viewed locally as a decomposition of $L$ into pairs of paths.  We make some notational preparation before giving the formal definition. 

Denote by $\Sigma \subset S^1$ those $x$-coordinates which coincide with a crossing or cusp of $L$.  We can write, $\displaystyle S^1 \setminus \Sigma = \bigsqcup_{m=1}^M I_m$ with each $I_m$ an open interval (or all of $S^1$ if $\Sigma = \emptyset$).  Making the convention that $I_0 = I_M$, we assume that the $I_m$ are ordered so that $I_{m-1}$ appears immediately to the left of $I_{m}$ and $I_M$ contains $x=0$.  On subsets of the form $I_m \times \R$ the front projection $L$ consists of some number of non-intersecting components which project homeomorphically onto $I_m$.  We refer to these components as the {\it strands} of $L$ above $I_m$, and we number them from \emph{top to bottom} as $1, \ldots, N(m)$.  Finally, for each $m = 1, \ldots, M$ we choose a point $x_m \in I_m$.

\begin{definition} \label{def:NR}
A {\it normal ruling} of the front diagram $L$ is a sequence $\rho = (\rho_1, \ldots, \rho_M)$ of involutions 
\[
\rho_m: \{1, \ldots, N(m) \} \rightarrow \{1, \ldots, N(m) \}, \quad (\rho_m)^2 = \mathit{id}
\]
satisfying the following restrictions:

\begin{enumerate}

\item  Each $\rho_m$ is fixed point free.

\item  If the strands above $I_{m}$ labeled $k$ and $k+1$ meet at a left cusp in the interval $(x_{m-1}, x_m)$, then $\rho_{m}(k) = k+1$ and when $n \notin \{k, k+1\}$, \[\rho_{m}(n) = \left\{\begin{array}{cr}\rho_{m-1}(n) & \mbox{if } n <k \\
\rho_{m-1}(n-2) & \mbox{if } n > k+1
\end{array}\right..\]

\item A condition symmetric to (2) at right cusps. 

\item If strands above $I_m$ labeled $k$ and $k+1$ meet at a crossing on the interval $(x_{m-1}, x_m)$, then $\rho_{m-1}(k) \neq k+1$ and either
\begin{enumerate}
\item $\rho_{m} = (k \,\, k+1) \circ \rho_{m-1} \circ (k \,\, k+1)$ where $(k \,\, k+1)$ denotes the transposition, or
\item $\rho_{m} = \rho_{m-1}$.
\end{enumerate}
In the second case we refer to the crossing as a {\it switch} of $\rho$.  Finally, we have a requirement at switches that is known as the {\it normality condition}.
\item  If there is a switch on the interval $(x_{m-1}, x_m)$ then one of the following three orderings holds:  
\[
\rho_m(k+1) < \rho_m(k) < k < k+1, \quad \rho_m(k) < k < k+1 < \rho_m(k+1), \, \mbox{or}
\]
\[
 \quad k < k+1 < \rho_m(k+1) < \rho_m(k)
\]

\end{enumerate}
\end{definition}

\begin{remark} \label{rem:AltNR} This definition is a slight variation on those found elsewhere in the literature.  Letting $\pi: S^1 \times \R \rightarrow S^1$ denote the projection, Chekanov and Pushkar defined a normal ruling as a continuous, fixed point free involution of $L\setminus \pi^{-1}(\Sigma)$ which preserves the $x$-coordinate and is subject to some requirements for continuous extension near crossings or cusps as well as a normality condition at switches.  Such an involution is recovered from our definition by viewing the set $\{1, 2, \ldots, N(m)\}$ that $\rho_m$ permutes as the set of strands above $I_m$.

From this perspective, the fixed point free condition causes the $\rho_m$ to divide the strands above $I_m$ into pairs, and in our figures we will present normal rulings by indicating this pairing.  Beginning at $x=0$ and working to the right, one may cover the front diagram with pairs of continuous paths with monotonically increasing $x$-coordinates, so that a given pair of paths corresponds to strands paired by the involutions.  If a path proceeds all the way around the annulus, then it will not necessarily end up where it started.  However, the division of the front diagram into pairs of points at $x=0$ and $x=1$ should match up.

Paired paths are only allowed to meet at common cusp endpoints.  In particular, at any crossing the two paths of the ruling that meet should belong to different pairs and, for values of $x$ near the crossing, each will have a ``companion path'' located somewhere above or below the crossing.  The two paths can either follow the link diagram and cross each other (this corresponds to (4) (a) above) or they may switch strands by each turning a corner at the crossing.  The normality condition provides a restriction on the location of the companion paths near a switch; out of six possible configurations for the switching strands and their companion strands only three are allowed.  See Figure \ref{fig:NormC} for the normality condition and the right half of Figure \ref{fig:Ex} for an example of a normal ruling.
\end{remark}

\begin{figure} 
\centerline{\includegraphics{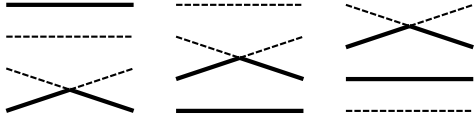}}
\caption{The normality condition.}
\label{fig:NormC}
\end{figure}

\subsection{Maslov potentials and graded normal rulings} \label{sec:MP} Further grading restrictions may be placed on a normal ruling after the introduction of a Maslov potential for $L$.   Let $p$ be a divisor of $2 r(L_i)$ for each component $L_i$ of a Legendrian link $L$.  A $\Z/p$-valued {\it Maslov potential}, $\mu$, for $L$ is a function from $L$ to $\Z/p$ which is constant except at cusp points where it increases by $1$ when moving from the lower strand to the upper strand. 
Note that a chosen orientation provides $L$ with a $\Z/2$-valued Maslov potential by following the convention that strands oriented to the right (resp. left) are assigned the value $0$ (resp. $1$) mod $2$.  

We say that a normal ruling $\rho$ is {\it $p$-graded} with respect to a $\Z/p$-valued Maslov potential $\mu$ if whenever two strands $S_1$ and $S_2$ of $L$ are paired by one of the $\rho_m$ with $S_1$ above $S_2$ we have $\mu(S_1) = \mu(S_2) + 1$.  

\subsection{Ruling polynomials}  Suppose $\mu$ is a $\Z/p$-valued Maslov potential for a Legendrian link $L$.  The $p$-graded ruling polynomial of $L$ with respect to $\mu$ is given by
\[\displaystyle
R^p_{(L,\mu)}(z) = \sum_{\rho} z^{j(\rho)}
\]
where the sum is over all normal rulings of $L$ which are $p$-graded with respect to $\mu$ and 
\[
j(\rho) = \#\mbox{switches} - \#\mbox{right cusps}.
\]

The ruling polynomial does not depend on the choice of Maslov potential when $p=1$; $p=2$ and $L$ is oriented; or $L$ is connected.  In any of these cases we denote the ruling polynomial simply as $R^p_L$.  The ruling polynomials are Legendrian isotopy invariants \cite{ChP}.

\subsection{Generalized normal rulings}  In the following definition the requirements from Definition \ref{def:NR} are relaxed in a manner which is appropriate for Theorem \ref{the:Main} to hold.

\begin{definition} \label{def:GNR} A {\it generalized normal ruling} consists of a sequence of involutions $\rho = (\rho_1, \ldots, \rho_M)$ as in Definition \ref{def:NR} subject to the following modifications:
\begin{enumerate}
\item We remove the requirement that the $\rho_m$ be fixed point free.

\item If a crossing occurs in the interval $(x_{m-1},x_m)$ between the $k$ and $k+1$ strands above $I_{m-1}$ with exactly one of these two strands a fixed point of $\rho_m$, then we  decide if the crossing is a switch precisely as in (4) of Definition \ref{def:NR}.
If the crossing is indeed a switch then we require the additional normality condition that either
\[
\rho_m(k) = k < k+1 < \rho_m(k+1) \quad \mbox{or} \quad \rho_m(k) < k < k+1 = \rho_m(k+1).
\]
(See Figure \ref{fig:GenNormC}.)
\end{enumerate} 
\end{definition}

\begin{figure} 
\centerline{\includegraphics{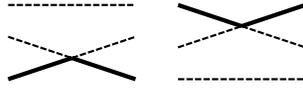} }
\caption{The normality condition for generalized rulings:  The strand pictured in bold is a fixed point of $\rho_m$.}
\label{fig:GenNormC}
\end{figure}

\begin{remark} \label{rem:GNR} (i) If a crossing involving the $k$ and $k+1$ strands occurs on $(x_{m-1}, x_m)$ with both of the crossing strands fixed by the ruling, i.e. $\rho_{m-1}(k) =k$ and $\rho_{m-1}(k+1) =k+1$, then $\rho_{m-1} = (k \, k+1) \circ \rho_{m-1} \circ (k \, k+1)$.  Consequently, we will not consider such crossings to be switches.

(ii) In the presence of an appropriate Maslov potential, we can consider $p$-graded generalized normal rulings precisely as in Section \ref{sec:MP}.

(iii) The number of generalized normal rulings of a Legendrian link is not invariant under Legendrian isotopy.  However, in view of Lemma \ref{lem:GNRiff} below, the polynomials $R^p_{L \cdot A_\lambda}$ serve as some form of substitute for a ``generalized ruling polynomial''.
\end{remark}

For establishing (1) of Theorem \ref{the:Main} we will use the following equivalent characterization of front diagrams that admit generalized rulings.

\begin{lemma}  \label{lem:GNRiff}
A front diagram $L$ has a $1$-graded (resp. $2$-graded) generalized normal ruling if and only if there exists  partitions $\lambda$ and $\mu$ so that $R^1_{L \cdot A_\lambda}(z) \neq 0$ (resp. $R^2_{L \cdot A_\lambda A_{-\mu}}(z) \neq 0$). 
\end{lemma}

\begin{proof}  For simplicity, we treat the $1$-graded case first.  If $R^1_{L \cdot A_\lambda}(z) \neq 0$, then $L\cdot A_\lambda$ has a normal ruling, $\rho$.  This produces a generalized normal ruling of $L$ by restricting $\rho$ to $L$ and treating any strands of $L$ which are paired with $A_\lambda$ as fixed point strands.  The normality condition from Definition \ref{def:GNR} follows from that of Definition \ref{def:NR}.  

Now suppose that $L$ has a generalized normal ruling.  If one of the $\rho_m$ has a fixed point strand, then we can continuously follow the fixed point strand around the diagram turning corners only at switches.  The result is a portion of the front diagram, $C_i$, without cusps that we suppose winds $\lambda_i$  times around the annulus.  There may be several fixed point components of this type.  We may assume the $\lambda_i$ are ordered so that they form a partition, $\lambda$.  The product $L\cdot A_\lambda$ has a normal ruling where each $C_i$ is paired with the component $A_{\lambda_i}$ of $\lambda$.  Such a ruling is completely determined once we specify the pairing between $C_i$ and $A_{\lambda_i}$ at a single point of $C_i$.  Now, the normality condition of Definition \ref{def:NR} follows from that of Definition \ref{def:GNR}, and the ordering of the factors of $A_\lambda$ is not important here since we do not have switches between any of the $C_i$ (Remark \ref{rem:GNR}).  See Figure \ref{fig:Ex}.

For the $2$-graded case, observe that in a $2$-graded ruling the orientation of strands meeting at a switch must agree.  Therefore, the $C_i$ each have a consistent orientation, and we choose an orientation on the component $A_{\lambda_i}$ accordingly.
\end{proof}

\begin{figure} 
\centerline{\includegraphics{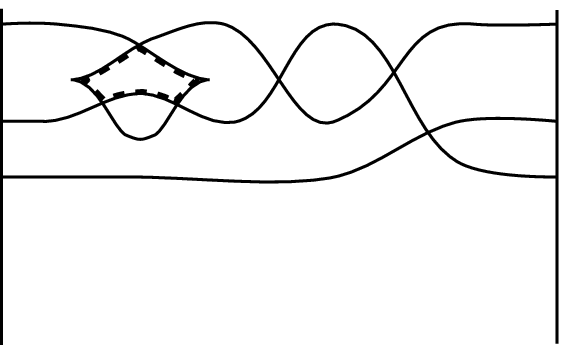} \raisebox{1.5cm}{$\rightarrow$} \includegraphics{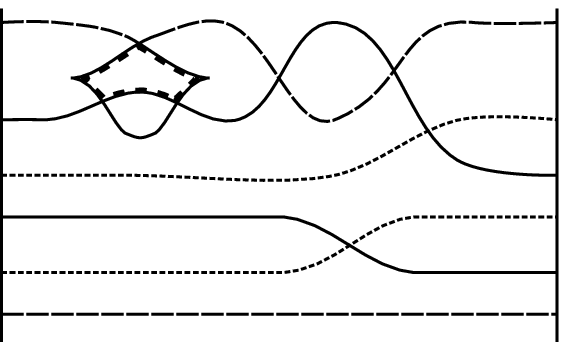}}
\caption{A generalized ruling with $3$ fixed point strands producing a normal ruling of $L \cdot A_\lambda$ with $\lambda = (2,1)$.}
\label{fig:Ex}
\end{figure}

\section{Kauffman polynomial and computation of $1$-graded ruling polynomials}

An analysis of how to compute $2$-graded ruling polynomials of Legendrian solid torus links from the HOMFLY-PT polynomial is done in \cite{R2}. In this section, we will perform a similar analysis of the $1$-graded case. We will derive formulas for the $1$-graded ruling polynomial of $A_\lambda$, and then relate the general case to a coefficient of an appropriate specialization of the Kauffman polynomial.

\subsection{Normal rulings of the product $A_\lambda$}

Given a front diagram $L$ with normal ruling $\rho$ we define the {\it decomposition} of $L$ with respect to $\rho$ as the Legendrian link, $L_\rho$,  obtained by resolving the switches of $L$ into parallel horizontal strands as \[
\fig{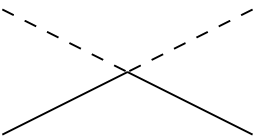} \rightarrow \fig{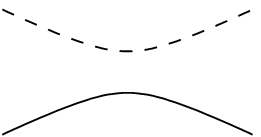}.
\]  The involutions of the strands of $L$ piece together to provide an involution, which we also denote as $\rho$, now defined on all of $L_\rho$.  The involution $\rho$ is continuous where we now view $L_\rho$ as a subset of $J^1(S^1)$ rather than just a front diagram, and its only fixed points correspond to the cusps of the front projection of $L_\rho$.  (Compare with Remark \ref{rem:AltNR}.)  The normal ruling of $L$ induces a normal ruling of $L_\rho$ where none of the crossings are switches.   

We record some observations about normal rulings of the products $A_\lambda$.

\begin{lemma} \label{lem:k} Suppose $\rho$ is a normal ruling of $L = A_\lambda$,

\begin{enumerate}
\item The decomposition, $L_\rho$, is also a product of basic fronts.
\item The involution $\rho$ must take a component of $L_\rho$ isotopic to $A_m$ to another component isotopic to $A_m$.
\item If components $C_1$ and $C_2$ of $L_\rho$ share a common switch of $L$, with $C_1$ above $C_2$ on the $z$-axis, then the vertical ordering of the four components $C_1$, $C_2$, $\rho(C_1)$, and $\rho(C_2)$ must be one of: 
\[
[\rho(C_2), \dots, \rho(C_1), \dots, C_1, C_2] ; \quad  [\rho(C_1), \dots, C_1, C_2, \dots, \rho(C_2)]; \mbox{ or} 
\]
\[
[C_1, C_2, \dots, \rho(C_2), \dots, \rho(C_1)].
\] 

\item The restriction of $\rho$ to a pair of components of $L_\rho$, $C_1$ and $C_2 =\rho(C_1)$, is completely determined by its value at a single point, $w \in C_1$.  Moreover, if $C_1 \cong A_m$ then there are precisely $m$ choices for $\rho(w) \in C_2$, and any one of them extends continuously to all of $C_1$. 
\item Two components of $L_\rho$ of the form $C_1$ and $\rho(C_1)$ cannot correspond to subsets of the same component of $L$.
\end{enumerate}
\end{lemma}

\begin{proof} Item (1) is clear; (2) follows from continuity of $\rho$; and (3) is a consequence of the normality condition.  The first assertion of (4) follows from continuity of $\rho$.  The second follows since $\rho(w)$ and $w$ must have the same $x$-coordinate and  $C_2$ also consists of $m$ strands.  That any such choice of $\rho(w)$ extends to all of $C_1$ is easily seen. 

We prove (5) by contradiction.
Suppose $C_1$ and $\rho(C_1)$ did come from the same component of $L$, and without loss of generality assume $\rho(C_1)$ is below $C_1$. They cannot meet at a switch as this would violate the normality condition.  Thus, there is some other component $C_2$ on the other end of the switch below $C_1$. The only possible position of $\rho(C_2)$ is then between $C_2$ and $\rho(C_1)$.  
Then $C_2$ and $\rho(C_2)$ also came from the same component of $L$. They cannot meet at a switch, so there is some further component $C_3$ immediately below $C_2$, which is paired with a component $\rho(C_3)$ between $C_3$ and $\rho(C_2)$.  We can continue this argument to produce arbitrarily many components of $L_\rho$ between $C_1$ and $\rho(C_1)$. 
\end{proof}

\subsection{Computing $R^1_{A_m A_m}$}

The results in the previous section are sufficient to compute the ruling polynomial for the simplest possible product, $A_m A_m$ (the ruling polynomial of a single basic front $A_m$ is $0$ by (5) of Lemma \ref{lem:k}).  Although this agrees with $R^2_{A_mA_{-m}}$ which is computed in Lemma 4.1 of \cite{R2}, the form of the answer given here is simplified and the proof is quite different.

\begin{lemma} 
\label{lem:am2}
The ruling polynomial of $L = A_m A_m$ is 
$$\sum_{k=0}^{m-1} {m+k \choose 2k+1} z^{2k}.$$
\end{lemma}
\begin{proof}

\begin{figure}
  \centering
  {\label{fig:bijection-1}\includegraphics {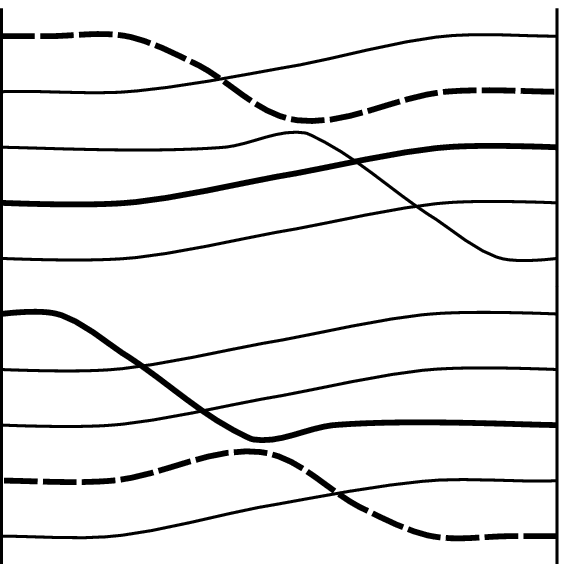}}  \quad
  {\label{fig:bijection-2}\includegraphics{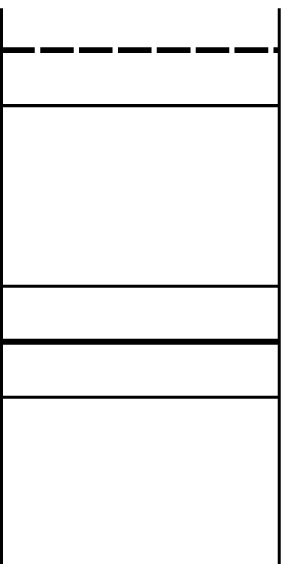}}  \quad
  {\label{fig:bijection-3}\includegraphics{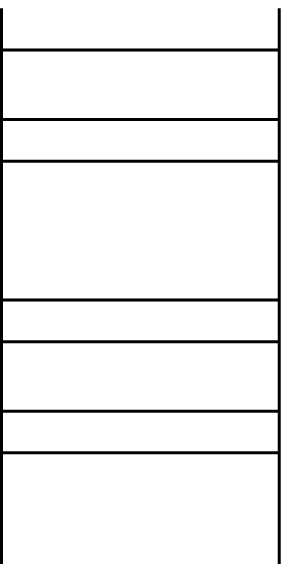}}
  \caption{The bijection between rulings of $A_5 A_5$ with 2 switches, divisions of 5 objects into 2 parts with a marked object in each part, and compositions of 7 into 4 positive parts.}
  \label{fig:bijection}
\end{figure}

 Normal rulings of $A_mA_m$ with $2k$ switches are in bijection with subdivisions of $m$ ordered objects into $k+1$ consecutive parts, with a marked object chosen in each part.

 The subdivision corresponds to choosing the location of $k$ switches within the first $A_m$ factor.  Specifically, dividing $m$ into parts $(\lambda_1, \ldots, \lambda_{k+1})$ corresponds to choosing $k$ switches so that in the decomposition, $L_\rho$, the first $A_m$ factor becomes $A_{\lambda_1}\cdots A_{\lambda_{k+1}}$.  In $L_\rho$, the $A_{\lambda_i}$ must be paired with $k+1$ components of the same size from the second $A_m$ factor, by Lemma~\ref{lem:k}~(2) and (5).  Then, Lemma~\ref{lem:k}~(3) determines the order of the components: they must be in the reverse order of the components from the first factor.  The total number of switches is $2k$.

The choice of marked object within a part $\lambda_i$ corresponds to choosing which strand within the $A_{\lambda_i}$ component is paired with the top strand of $\rho(A_{\lambda_i})$ at $x=0$.  These choices may be arbitrary, and they uniquely determine a ruling by Lemma~\ref{lem:k}~(4).  See Figure \ref{fig:bijection}.

To complete the proof observe that subdivisions of this type are in turn in bijection with compositions of $m+(k+1)$ into $2(k+1)$ positive parts,  $(a_1, b_1, \ldots, a_{k+1}, b_{k+1})$: two consecutive parts of size $a_i$ and $b_i$ correspond to a part $\lambda_i = a_i+b_i-1$ with the $a_i$-th object marked in a subdivision of $m$.  The number of ways to decompose $m+k+1$ objects into $2(k+1)$ parts of positive size is well-known: it is ${(m+k+1) - 1  \choose 2(k+1) - 1}$ or ${m+k \choose 2k+1}$.  This gives us the sum for the ruling polynomial.

\end{proof}

This formula will be used in the next section, so we will write $\< m \>$ for the ruling polynomial $R_{A_m A_m}(z)$, following the convention in \cite{R2}.

\subsection{A Formula for Arbitrary Products of Basic Fronts}

We will use the formula for $\< m \>$ to calculate the ruling polynomial of $A_\lambda$ for an arbitrary $\lambda$.

Given a normal ruling $\rho$ of $L = A_\lambda$, define the \emph{block} $B_{ij}$ to consist of those components of the decomposition $L_\rho$ which originated in the $i$-th component of $L$, and are paired by $\rho$ with components that originated in the $j$-th component of $L$. The size of the block, $b_{ij}$, is the number of points in $B_{ij}$ with some fixed $x$-coordinate, away from crossings.

\begin{lemma}
\label{lem:block}
Given a normal ruling of $L = A_\lambda= A_{\lambda_1} A_{\lambda_2}\cdots A_{\lambda_n}$, the blocks in the $i$-th component of $L$ consist of vertically consecutive components of $L_\rho$, and are themselves vertically ordered as follows (some blocks may be empty):
$$ B_{i,i-1} B_{i,i-2} \cdots B_{i,1} B_{i,n} B_{i,n-1} \cdots B_{i,i+1} $$
\end{lemma}
\begin{proof}
Suppose that when we resolve $A_{\lambda_i}$ at switches, we get the components $C_1, C_2, \dots, C_k$, in that vertical order. If, for some $j$, $\rho(C_j)$ is above $C_j$, then the normality condition demands that $\rho(C_{j-1})$ is between $\rho(C_j)$ and $C_{j-1}$. Similarly, if $\rho(C_j)$ is below $C_j$, then $\rho(C_{j+1})$ must be between $C_{j+1}$ and $\rho(C_j)$.

As a result, if $\rho(C_{j_1})$ and $\rho(C_{j_2})$ come from the same component of $L$, then $\rho(C_j)$ for $j_1 \le j \le j_2$ are between $\rho(C_{j_1})$ and $\rho(C_{j_2})$. This implies each block is made up of some number of consecutive components. And due to the normality condition, the ordering of any two consecutive blocks must be either $B_{i,j+1} B_{i,j}$, with $j>i$, or $B_{i,j-1} B_{i,j}$, with $j<i$ (with the caveat that some of the blocks may be empty, if $\rho$ does not pair two components of $L$ at all). Putting this together yields the block ordering above.\end{proof}

This means that once we pick the sizes of the blocks $b_{i,1} \dots b_{i,n}$, the locations of the blocks are determined. To complete the calculation of the ruling polynomial, observe that the choice of a normal ruling of the blocks $B_{ij}$ and $B_{ji}$, with $b_{ij} = b_{ji} = m$, is equivalent to the choice of a normal ruling of $A_m A_m$.

\begin{theorem}
\label{thm:front}
Let $\<m\>$ denote the ruling polynomial of $A_m A_m$, with $\<0\>$ taken to be $z^{-2}$. Then the ruling polynomial of $A_\lambda = A_{\lambda_1} A_{\lambda_2} \dots A_{\lambda_n}$ is given by
$$z^{n(n-1)} \sum_{(b_{ij}) \in M_{\lambda}} \prod_{i < j} \< b_{ij} \>$$
where $M_{\lambda}$ is the set of all symmetric matrices $(b_{ij})$ with nonnegative integer entries such that the row sums $\sum_{j=1}^n b_{ij} = \lambda_i$ and the trace $\op{tr}\,(b_{ij}) = 0$.
\end{theorem}
\begin{proof}

The choice of a matrix in $M_\lambda$ is equivalent to the choice of block sizes $b_{ij}$. By Lemma \ref{lem:block}, this also fixes the locations of the blocks.  A normal ruling of $A_\lambda$ is then completely determined by its restriction to pairs of blocks $B_{ij}$ and $B_{ji}$.

If the block size $b_{ij}$ is nonzero, then $\< b_{ij}\>$ describes the possible restrictions of the normal rulings to the union  $B_{ij}\cup B_{ji}$. We take the product to combine these normal rulings, but we have to account for the switches between the blocks. If all block sizes are nonzero, then there will be $n-1$ switches in each of the $n$ components of $L$, giving us a factor of $z^{n(n-1)}$. Any block $B_{ij}$ of size $0$ will reduce this number by $1$ in component $j$, but the corresponding block $B_{ji}$ will reduce the number of switches by $1$ in component $i$; this gives a factor of $z^{-2}$ which is accounted for by the convention of $\<0\> = z^{-2}.$
\end{proof}

\begin{corollary}
\label{cor:comm}
The 1-graded ruling polynomial is commutative in front diagram products: that is, the ruling polynomials of 
$$ A_{\lambda_1} A_{\lambda_2} \cdots A_{\lambda_i} A_{\lambda_{i+1}} \cdots A_{\lambda_n}$$
and
$$ A_{\lambda_1} A_{\lambda_2} \cdots A_{\lambda_{i+1}} A_{\lambda_i} \cdots A_{\lambda_n}$$
are equal.
\end{corollary}
\begin{proof}
There is an easy bijection between the possibilities for the matrix $M_{\lambda}$ and the new matrix $M_{\lambda'}$: we simply exchange the $i$-th and $(i+1)$-th columns and rows; the summands $\prod_{i < j} \< b_{ij} \>$ do not change.
\end{proof}

\subsection{Calculating the ruling polynomial from the Kauffman polynomial}

In $\R^3$, the 1-graded and 2-graded ruling polynomial of arbitrary Legendrian links may be easily recovered from the Kauffman and HOMFLY-PT polynomials. The second author shows in \cite{R2} that the $1$-graded (resp. $2$-graded) ruling polynomial of a link $L$ is the coefficient of $a^{-tb(L)}$ in the Kauffman polynomial (resp. HOMFLY-PT polynomial) of $L$.
In the case of Legendrian solid torus links we first need to specialize the extra variables in a non-multiplicative manner.

Using the notation of Section \ref{sec:KP}, consider the $R$-module homomorphism $\Psi: \mathcal{F} \cong R[A_1, A_2, \ldots] \rightarrow R$ determined by $A_\lambda \mapsto R^1_{A_\lambda}(z)$ when $\lambda$ is a partition.  (Compare with Remark \ref{rem:basis}.)  Given a link diagram $L$, we let $\widehat{D}_L(a,z) = \Psi( D_L)$, and $\widehat{F}_L(a,z) = a^{-w(L)} \widehat{D}_L(a,z)$.

\begin{theorem}
\label{thm:spec}
Let $L \subset J^1(S^1)$ be any Legendrian solid torus link.  Then, the $1$-graded ruling polynomial $R^1_L(z)$ is equal to the coefficient of $a^{-\tb(L)}$ in $\widehat{F}_L(a,z)$.
\end{theorem}

This result is analogous to Theorem 6.3 of \cite{R2}, where it is shown that we can recover the $2$-graded ruling polynomial from such a specialization of the HOMFLY-PT polynomial. The proof, via induction on a certain measure of complexity of a front diagram, carries through in the 1-graded case as well.  The base case consists of all products of basic fronts where the result follows from the crucial Corollary~\ref{cor:comm}.  Next, it is observed that the ruling polynomial and the coefficient of $a^{-\tb(L)}$ in $\widehat{F}_L$ share common skein relations which are Legendrian analogs of equations (\ref{eq:SR1})-(\ref{eq:SR3})  (see \cite{R1} or \cite{R2}).  Then, just as in \cite{R2}, the inductive step is completed by an algorithm which uses these skein relations to evaluate the invariants in terms of front diagrams of lesser complexity.

\begin{example}
Consider the Legendrian knots $L_1$ and $L_2 = L_1\cdot A_2 A_1$ pictured in Figure \ref{fig:Ex}, and suppose orientations are chosen so that all strands are oriented to the right when they pass through the vertical line $x=0$.  The Kauffman polynomials are given by
\[\begin{array}{ccl}
F_{L_1}= & &A_1 \times \left[ a^{-1}(-z-z^3) + a^{-2}z^4 + a^{-3}(z+2z^3) + a^{-4}z^2 \right] \\
& + & A_3 \times \left[ a^{-1}(z+z^3) + a^{-2}(-z^2-z^4) + a^{-3}(-z-z^3) \right] \\
& + & A_2 A_1 \times \left[ a^{-1}(1+z^2) - a^{-2}z^3 - a^{-3}z^2 \right],
\end{array}
\]
and $F_{L_2} = a^{-1} A_2A_1 F_{L_1}$.  We have $\tb(L_1) = 1$ and $\tb(L_2) =2$, so in both cases the estimate (\ref{eq:CGEst}) is sharp.

Using Theorem \ref{thm:front}, one has $R^1_{A_{(2,1,1)}}(z)= z$; $R^1_{A_{(3,2,1)}}(z)=  2z+z^3$; and $R^1_{A_{(2,2,1,1)}}(z)=  2+3z^2$.  This allows us to compute 
\[
\widehat{F}_{L_2} = a^{-2}( 2 + 6 z^2 + 5 z^4 + z^6) +a^{-3}(-4 z^3 - 5z^5 - z^7) + a^{-4} (-3z^2-4z^4-z^6) +a^{-5} z^3,
\]
and Theorem \ref{thm:spec} gives $R^1_{L_2}(z) = 2 + 6 z^2 + 5 z^4 +z^6$ which can be verified directly.
\end{example}

\section{Generalized normal rulings and the Thurston-Bennequin estimates}

In this section we establish the equivalence (1) of Theorem \ref{the:Main} which follows from Lemma \ref{lem:GNRiff} together with the following:

\begin{theorem}
\label{thm:sharp}
Let $L$ be a Legendrian link in the solid torus. Then the equality
$$tb(L) = -\deg_a F_L$$
holds if and only if there exists a partition $\lambda$ so that $L \cdot A_\lambda$ has a normal ruling. 
\end{theorem}

\begin{proof}{(Theorem~\ref{thm:sharp})}
One direction is straightforward. Suppose that, for some $\lambda$, $L' = L \cdot A_\lambda$ has a normal ruling. Then the ruling polynomial of $L'$ is nontrivial, so the coefficient of $a^{-tb(L)}$ is nonzero. Therefore $tb(L') \ge - \deg_a F_{L'}$ which, combined with the inequality (\ref{eq:CGEst}), gives us an equality $\tb(L') = - \deg_a F_{L'}$. However, $tb(L') = tb(L \cdot A_\lambda) = tb(L) + w(A_\lambda) $, since $A_\lambda$ has no cusps. In addition, $D_{L'} = A_\lambda \cdot D_L$, so $F_{L'} = a^{-w(A_\lambda)}A_\lambda \cdot F_L$, and we compute
\[
-\deg_a F_L = -w(A_\lambda) - \deg_a(F_{L'}) = -w(A_\lambda) + \tb(L') = \tb(L).
\]

Now suppose $tb(L) = -\deg_a F_L$.  We will find a $\lambda$ such that $L \cdot A_\lambda$ has a normal ruling.

Let $\sum_{\mu} p_\mu(z) A_\mu$ be the coefficient of $a^{-tb(L)}$ in $F_L$, where the $p_\mu(z)$ are polynomials in $z$ and $z^{-1}$. This coefficient is nonzero, or else the degree equality would not hold, so $p_\mu(z) \ne 0$ for at least one $\mu$. Let $k$ be the smallest integer such that at least one $p_\mu$ has a nonzero coefficient of $z^k$.

By Theorem~\ref{thm:spec}, the ruling polynomial of $L \cdot A_\lambda$ is 
$$\sum_\mu p_\mu(z) R_{A_\mu A_\lambda}(z).$$
We will prove that for some $\lambda$, this polynomial is nonzero (and therefore a normal ruling exists) by looking at the $z^k$ coefficient of this polynomial. Since $R_{A_\mu A_\lambda}(z)$ is a polynomial in $z$ with no terms of $z^{-1}$ or lower degree, the only way to get a $z^k$ coefficient is from the product of $p_\mu(z)[z^k]$ and $R_{A_\mu A_\lambda}(z)[z^0]$ for some $\mu$ (here, $f(z)[z^i]$ denotes the coefficient of $z^i$ in $f(z)$). Denote $p_\mu(z)[z^k]$ by $a_\mu$, and $R_{A_\mu}(z)[z^0]$ (which is the number of switchless rulings of $A_\mu$) by $C(\mu)$.

The quantity $C(\mu)$ is easy to calculate. Without switches, each component of size $k$ must simply be paired with another component of size $k$ in one of $k$ ways. In particular, this is only possible if there is an even number of each component size. Define the double factorial $(2k-1)!! = (2k-1)(2k-3)(\cdots)(3)(1) = (2k)!/(2^k k!)$. This counts the number of ways to divide $2k$ objects into pairs. It is clear that
$$
C(\mu) = \begin{cases}
\prod_{k=1}^n k^{a_k} (2a_k-1)!! & \mbox{if }\mu = 1^{2a_1} 2^{2a_2} \dots n^{2a_n} \\
0 & \mbox{else}
\end{cases}
$$

We wish to prove that for some $\lambda$, $\sum_{\mu} a_\mu C(\mu \cdot \lambda) \ne 0$. Here, if $$\mu = 1^{a_1} 2^{a_2} \dots n^{a_n}\mbox{ and }\lambda = 1^{b_1} 2^{b_2} \dots n^{b_n},$$ we will denote by $\mu \cdot \lambda$ the partition $$1^{a_1 + b_1} 2^{a_2 + b_2} \dots n^{a_n + b_n}.$$

Let $M$ be the collection of all partitions such that 
\begin{enumerate}
\item The parts of the partition are all no larger than $n$, for some $n$.
\item Parts of each size occur between $0$ and $2m-1$ times, for some $m$.
\end{enumerate}
We choose the parameters $m$ and $n$ such that we include all partitions $\mu$ with $a_\mu \ne 0$. 

Let $V$ be a $n^{2m}$-dimensional real vector space with basis vectors $e_{\lambda}$ for $\lambda \in M$. For each $\mu \in M$, consider the following vectors in $V$:

$$v_\mu = \sum_{\lambda \in M} C(\mu \cdot \lambda) e_\lambda.$$

We will show that these vectors also form a basis of $V$, and are therefore linearly independent. From there, observe that 
$$\sum_{\lambda \in M} \left(\sum_{\mu \in M} a_\mu C(\mu \cdot \lambda) \right) e_\lambda 
= \sum_{\mu \in M} a_\mu \left(\sum_{\lambda \in M} C(\mu \cdot \lambda) e_\lambda \right)
= \sum_{\mu \in M} a_\mu v_\mu.$$
If the coefficients $a_\mu$ on the right are not all $0$, then because the $v_\mu$ are linearly independent the resulting sum is a nonzero vector of $V$. Therefore the coefficients in terms of $e_\lambda$ are not all $0$ as well -- i.e. for some $\lambda$, $\sum_{\mu} a_\mu C(\mu \cdot \lambda) \ne 0$. So once we have the result of linear independence, we are done.

From the formula for $C(\mu)$, it's easy to calculate that $C(\mu \cdot \lambda)$ can be written as a product of $C(k^{a_k} \cdot k^{b_k})$, over all $k$, where $a_k$ and $b_k$ are the number of parts of size $k$ in $\lambda$ and $\mu$ respectively. Suppose we write $V$ as the tensor product $\bigotimes_{i=1}^n \R^{2m}$, identifying the basis vector $e_{j_1} \otimes e_{j_2} \otimes \cdots \otimes e_{j_n}$ on the left with the basis vector $e_\lambda$ on the right, where $\lambda = 1^{j_1} 2^{j_2} \cdots n^{j_n}$. Here we use a slightly non-standard basis of $\R^{2m}$: it is $0$-indexed and consists of $\{e_0, e_1, \dots, e_{2m-1}\}$, for ease of notation. 

Then if $\mu = 1^{a_1} 2^{a_2} \cdots n^{a_n}$,
\begin{align*}
v_\mu &= \sum_{\lambda \in M} C(\mu \cdot \lambda) e_\lambda \\
 &= \sum_{1^{b_1} \cdots n^{b_n} \in M}  \left( \prod_{i=1}^n C(i^{a_i} \cdot i^{b_i}) \right) \left( \bigotimes_{i=1}^n e_{b_i}\right) \\
 &= \sum_{1^{b_1} \cdots n^{b_n} \in M}  \left( \bigotimes_{i=1}^n  C(i^{a_i} \cdot i^{b_i}) e_{b_i} \right) \\
 &= \bigotimes_{i=1}^n \left( \sum_{j=0}^{2m-1} C(i^{a_i} \cdot i^j) e_j \right).
\end{align*}

Therefore, rather than prove that the vectors $v_\mu$ are a basis of $V$, it suffices to prove that the vectors $u_k = \sum_{j=0}^{2m-1} C(i^{k} \cdot i^j) e_j$, as $k$ goes from $0$ to $2m-1$, are a basis of $\R^{2m}$.
There are three simplifying observations to be made:
\begin{enumerate}
\item $C(i^{k} \cdot i^j) = 0$ if $k \not\equiv j \pmod{2}$. Therefore $u_k$ is a linear combination only of the odd-indexed $e_j$ if $k$ is odd, and only of the even-indexed $e_j$ if $k$ is even. 
Furthermore, $C(i^k \cdot i^j) = C(i^{k-1} \cdot i^{j+1})$, so $u_{2k}$ and $u_{2k-1}$ have the same coefficients, just shifted over by one index. As a result, we will only show the independence of the vectors $u_0, u_2, \dots, u_{2m-2}$ -- the result for $u_1, u_3, \dots, u_{2m-1}$ is similar.

\item By the first observation, we have 
$$u_{2k}= \sum_{j=0}^{m-1} C(i^{2k} \cdot i^{2j}) e_{2j} = i^k \sum_{j=0}^{m-1} C(1^{2k} \cdot 1^{2j}) (i^j e_{2j}).$$
This corresponds to starting in the case where $i=1$, then scaling both the $u_{2k}$ and the $e_{2j}$ by powers of $i$ -- a scaling which doesn't change the question of linear independence one way or the other. Therefore it suffices to consider the case where $i=1$.

\item Finally, we can scale each $u_{2k}$ by $C(1^{2k})$ (which, too, doesn't affect linear independence). Now we want to look at 
$$u_{2k}' = \sum_{j=0}^{m-1} C(1^{2k} \cdot 1^{2j})/C(1^{2k}) e_{2j} = \sum_{j=0}^{m-1} \left(\prod_{\ell=1}^j (2k+2\ell-1)\right) e_{2j}.$$
\end{enumerate}

If we put the coefficients of $u_{2k}'$ as columns of a matrix, (i.e. $j$ indexes the rows and $k$ indexes the columns), we get
$$
\begin{pmatrix}
1 & 1 & \dots & 1 \\
1 & 3 & \dots & 2m-1 \\
1\cdot 3 & 3\cdot 5 & \dots & (2m-1)(2m+1) \\
\vdots & \vdots & \ddots & \vdots \\
1\cdot 3 \cdots (2m-1) & 3 \cdot 5 \cdots (2m+1) & \dots & (2(m-1)+1)(\cdots)(4(m-1)-1)
\end{pmatrix}
$$
Here, the entries in the $j$-th row are given by $f_j(k) = \prod_{\ell=1}^j (2k+2\ell-1)$, which is a degree $j$ polynomial function. In particular, $f_j(k)$ can be written as $(2k)^j$ plus lower-order terms; these lower-order terms are necessarily a linear combination of $f_1(k), \dots, f_{j-1}(k)$. Therefore, we can use row operations to eliminate the lower-order terms, so that the resulting matrix is:
$$
\begin{pmatrix}
1 & 1 & \dots & 1 \\
1 & 2 & \dots & m \\
1 & 4 & \dots & m^2 \\
\vdots & \vdots & \ddots & \vdots \\
1 & 2^{m-1} & \dots & m^{m-1}
\end{pmatrix}
$$
This is a Vandermonde matrix whose determinant is $\prod_{j \ne k} (j -k)$, which is nonzero. Therefore the vectors $u_{2k}'$ (and $u_{2k}$) form a basis of $\R^{2m}$, which completes the proof.
\end{proof}

\subsection{The $2$-graded case and the HOMFLY-PT estimate}  A similar approach applies in the case of the HOMFLY-PT polynomial, $P_L$.  The proof of the  reverse implication is identical.  For the forward implication, we suppose $\tb(L) = -\deg_aP_L$ and consider the coefficient of the lowest power $z^k$ that appears in the $a^{-\tb(L)}$ term of $P_L$,
 \[
 \sum_{\alpha, \beta} b_{(\alpha, \beta)} A_\alpha A_{-\beta}.
 \]
 Fix parameters $m$ and $n$ so that the set 
 \[
 M= \left\{ (\mu, \nu) \,|\, \mu = 1^{a_1}\cdots n^{a_n}, \, \nu= 1^{b_1}\cdots n^{b_n}, \, 1\leq a_i, b_i\leq m \right\}
 \]
contains all $(\alpha, \beta)$ such that $b_{(\alpha, \beta)} \neq 0$.

Using Theorem 6.3 in \cite{R2}, for any $(\mu, \nu) \in M$ the coefficient of $z^k$ in the $2$-graded ruling polynomial of $L \cdot A_\mu A_{-\nu}$ is given by
\[
  \sum_{\alpha, \beta} b_{(\alpha, \beta)} R^2_{A_{\alpha\cdot\mu} A_{-\beta\cdot\nu}}(0).
\]
It suffices to show that the coefficient matrix 
\[\displaystyle
A= \left(R^2_{A_{\alpha\cdot\mu} A_{-\beta\cdot\nu}}(0)\right)_{(\alpha, \beta), (\mu, \nu) \in M}
\]
is non-singular.  Writing $\alpha =1^{a_1}\cdots n^{a_n}$, $\beta= 1^{b_1} \cdots n^{b_n}$, $\mu = 1^{c_1} \cdots n^{c_n}$, $\nu = 1^{d_1}\cdots n^{d_n}$, one has
\[
R^2_{A_{\alpha\cdot\mu} A_{-\beta\cdot\nu}}(0) = \prod_{k=1}^n \delta_{a_k +c_k, b_k+d_k} k^{a_k+c_k}(a_k + c_k)!.
\]
Thus, $A$ is a tensor product (Kronecker product) of matrices 
\[
A_k= \left( \delta_{a+c, b+d} k^{a+c} (a+c)!\right)_{(a,b), (c,d)}.
\]
Due to the Kronecker delta, each $A_k$ is a direct sum (block matrix) of matrices $B_l,$ $ l \in \Z\cap[-n,n]$ obtained from keeping rows and columns satisfying $a-b = d-c = l$.  

The proof is completed by showing that each $B_l$ is non-singular.  We treat the case $l \geq 0$ as $l <0$ is similar.  Then, $ l \leq a,d \leq n$ and $B_l = \left( k^{a+d-l}(a+d-l)!\right)$.  Dividing rows and columns by $k^{a-l}$ and $k^{d}\cdot d!$ respectively leaves
\[
\left((a+d-l)!/d! \right) = \left( f_a(d) \right),
\] 
where $f_a(x) = \prod_{j=1}^{a-l} (j +x)$ is a polynomial of degree $a-l$.  Elementary row operations reduce this to a non-singular Vandermonde matrix.

\section{Augmentations and generalized normal rulings}

In this final section we complete the proof of Theorem \ref{the:Main} by establishing the following:

\medskip

For any Legendrian link $L \subset J^1(S^1)$ with $\Z/p$-graded Maslov potential, $\mu$, the following are equivalent: 
\begin{itemize}
\item[{\bf (A)}] The Chekanov-Eliashberg algebra, $(\mathcal{A}(L), d)$, admits a $p$-graded augmentation.
\item[{\bf (B)}] The front projection of $L$ has a $p$-graded generalized normal ruling.
\end{itemize}

\medskip

We begin by briefly recalling the aspects of the Chekanov-Eliashberg DGA that are important for the proof.  The reader is refered to \cite{NgTr} for the original, more detailed treatment of this DGA in the $J^1(S^1)$ setting.

Given a Legendrian knot or link, $L \subset J^1(S^1)$, the {\it Lagrangian projection}, $\pi_{xy}(L)$, of $L$ to the $xy$-annulus is an immersed curve.  The Chekanov-Eliashberg DGA, $(\A(L), d)$, is a graded algebra, $\A(L)$, with a degree $-1$ differential, $d$, defined via a generic Lagrangian projection of $L$.  

After a small Legendrian isotopy, we may assume $\pi_{xy}(L)$ to have only finitely many transverse double points which we label as $q_1, \ldots, q_n$.
Then, the algebra, $\A(L)$, is the free associative $\Z/2$-algebra with unit generated by the double points $q_1, \ldots, q_n$.  The set of monic non-commutative monomials in the $q_i$ forms a linear basis for $\A(L)$.  If $L$ is connected, then $\A(L)$ has a $\Z/(2 r(L))$ grading.  In general, the grading depends on a choice of Maslov potential for $L$.  The differential, $d$, is defined by counting certain immersed discs in the $xy$-annulus with boundary mapped to the Lagrangian projection of $L$.  

\begin{definition}  An {\it augmentation} of $(\A(L), d)$ is an algebra homomorphism $\varepsilon: \A(L) \rightarrow \Z/2$ satisfying
\begin{itemize}
\item[(i)] $\varepsilon(1) = 1$, and
\item[(ii)] $\varepsilon \circ d = 0$.
\end{itemize}
In addition, $\varepsilon$ is {\it $p$-graded} if $\varepsilon(q_i) \neq 0$ implies $|q_i| = 0 $ mod $p$.
\end{definition}

The existence of an augmentation of $(\A(L), d)$ is a property that is invariant under Legendrian isotopy.  This follows from the fact that the ``stable tame isomorphism type'' (see \cite{Ch}, \cite{NgTr}) of $(\A(L), d)$ is unchanged by a Legendrian isotopy.  Therefore, in establishing the equivalence of {\bf (A)} and {\bf (B)} we may work with the Chekanov-Eliashberg algebra of a Legendrian isotopic link $L'$.  The links $L'$ which we will consider have a standard form so that $(\A(L'), d)$ may be described in a formulaic manner from the front projection of $L'$ (and this front projection is combinatorially the same as that of $L$).  For this reason we do not present the differential or the grading of the Chekanov-Eliashberg DGA in full generality here.

\subsection{The DGA of a resolved front diagram with splashes}

Given a Legendrian $L \subset J^1(S^1)$ we begin by modifying the front diagram of $L$ via (a slight variation of) the resolution technique of \cite{NgTr}.  Beginning near $x=0$ and working from left to right, we alter the front projection of $L$ by an isotopy in the $xz$-annulus as follows.  We arrange so that, except for intervals  near $x=1$ or immediately prior to a crossing or right cusp, the slopes of the strands are constant and strictly decreasing as we move from the top to bottom.  Further, we will assume that all strands usually have non-positive slope.  It is no problem to produce these conditions after a left cusp, but with crossings and right cusps the slopes of the two relevant strands will need to be interchanged prior to the crossing or cusp.  As the $y$-coordinate is given by the slope $\frac{dz}{dx}$, this has the effect of producing double points on the Lagrangian projection corresponding to (but located to the left of) the crossings and right cusps of the front projection of $L$.  Finally, when we near $x=1$ the strands have become very spread out and moved below their original $z$ values at $x=0$.  Beginning with the top strand and then proceeding successively to the lowest strands, we return each strand back to its initial position via a steep upward step.  This creates several new crossings on the Lagrangian projection.  See Figure \ref{fig:res}.  

\begin{figure} 
\centerline{\includegraphics{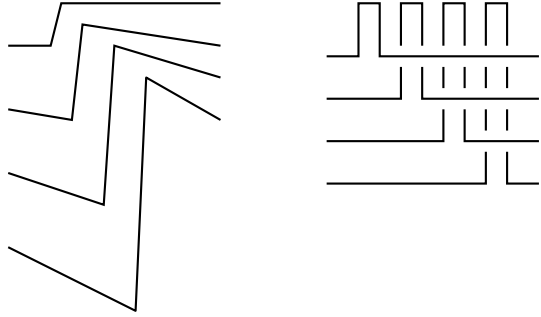}}
\caption{The front projection (left) and Lagrangian projection (right) of $L'$ in an interval immediately to the left of $x=1$.}
\label{fig:res}
\end{figure}

Next, we add ``splashes''.  Recall that we view the $S^1$ factor of $J^1(S^1)$ as $[0, 1]$ with $0$ and $1$ identified.  In a similar notation to Section 2, we let $0 = x_0 < x_1 < \cdots < x_M =1 $ be a partitioning of the interval $[0,1]$ such that no $x_m$ coincides with the $x$-coordinate of a crossing or cusp and each interval $(x_{m-1}, x_m)$ contains exactly one crossing or cusp.  For each $m = 1, \ldots, M-1$, we add a miniature version of the steps appearing in the part of the resolution procedure near $x=1$ into a small interval centered at $x_m$.  That is, beginning at the top strand and then working downward add a brief but steep (smooth) upward step into the diagram.  This has a minimal effect on the front projection but alters the Lagrangian projection at each $x_m$ by replacing what had been several parallel lines with a collection of crossings  similar to those pictured in the right half of Figure \ref{fig:res}.  Denote the Legendrian link resulting from the combination of these two procedures as $L'$.

We now give a complete description of the Chekanov-Eliashberg DGA of $L'$.  For each $1 \leq m \leq M$, let $N(m)$ denote the number of intersection points of $L$ with the plane $x=x_m$.  The generators of $\mathcal{A}(L')$ come from two sources.  First, we have generators corresponding to the crossings and right cusps of the front projection of $L$ via the resolution procedure.  In addition, for each $1 \leq m \leq M$ we have two upper triangular matrices worth of generators, $x^m_{ij}$ and $y^m_{ij}$ with $ 1 \leq i < j \leq N(m)$.  These correspond to the double points created by the splashes and the final step of the resolution procedure.  

\subsubsection{The grading}  If $L$ is equipped with a $\Z/p$-graded Maslov potential, $\mu$, then $\mathcal{A}(L')$ is $\Z/p$-graded.  We will describe the degree $|q_i|\in \Z/p$ assigned to the generators of $\mathcal{A}(L')$; degrees then extend additively as $|x\cdot y|= |x| + |y|$.

In the following, $\mu(m , i)$ denotes the value of the Maslov potential on the $i$-th strand at $x_m$.  (As in Section 2, we label strands from top to bottom.)  The generators of $\mathcal{A}(L')$ coming from splashes have the following degrees:
\begin{equation} \label{eq:degrees}
|x^m_{ij}| = \mu(m, i) - \mu(m,j), \quad \mbox{and} \quad   |y^m_{ij}| = \mu(m,i) - \mu(m,j) -1.
\end{equation}
In addition, a crossing $b_m$ between the $k$ and $k+1$ strands occurring in the interval $(x_{m-1}, x_m)$  has $|b_m| = \mu(m, k+1) - \mu(m,k)$, and all right cusps have degree $1$.

\subsubsection{The differential} Formulas for the differential $d$, are most efficiently provided by, for each $m$, placing the generators  $x^m_{ij}$ and $y^m_{ij}$ into strictly upper triangular matrices
\[
X_m = (x^m_{ij}), Y_m = (y^m_{ij}).
\]
(Here, $x^m_{ij} = y^m_{ij} = 0$ if $i \geq j$.)  As the $x$-coordinate is $S^1$-valued, it is important to make the convention that $X_0 = X_M$ and $Y_0=Y_M$.  Then, applying the differential to each entry, we have the formulas
\begin{equation} \label{eq:diffe}
\displaystyle
\begin{array} {ccl}
d Y_m & = & (Y_m)^2 \\
d X_m & = & Y_m (I + X_m) + (I + X_{m-1}) \widetilde{Y}_{m-1}
\end{array}
\end{equation} 
with $I$ an identity matrix of the appropriate size.  The precise form of $\widetilde{Y}_{m-1}$ depends on the tangle appearing on the interval $(x_{m-1}, x_m)$ and is described presently.

Suppose that $(x_{m-1}, x_m)$ contains a crossing, $b_m$, between the strands labeled $k$ and $k+1$.  Then, 
\[
d b_m = y^{m-1}_{k, k+1},\]
and
\[
\widetilde{Y}_{m-1} = B_{k,k+1} \widehat{Y}_{m-1} B^{-1}_{k, k+1}
\]
where  $ B_{k,k+1}$ (resp.  $B^{-1}_{k,k+1}$) agrees with the identity matrix except for a $2\times 2$ block $\left[\begin{array}{cc} 0 & 1 \\ 1 & b_m \end{array}\right]$ (resp. $\left[\begin{array}{cc} b_m & 1 \\ 1 & 0 \end{array}\right]$) along the diagonal in rows $k$ and $k+1$ and $\widehat{Y}_{m-1}$ is the matrix $Y_{m-1}$ with the entry $y^{m-1}_{k,k+1}$ replaced with $0$.

Next, we suppose  $(x_{m-1}, x_m)$ contains a single left cusp between the strands labeled $k$ and $k+1$ at $x_m$.  Then,
\[
\widetilde{Y}_{m-1} =J_k Y_{m-1} J^{\mathrm{T}}_{k} + E_{k,k+1}
\]
where $J_k$ is the $N(m) \times N(m)$ identity matrix with columns $k$ and $k+1$ removed and $E_{k,k+1}$ is a matrix with a single non-zero entry in the $k,k+1$ position.

Finally, we suppose  $(x_{m-1}, x_m)$ contains a single right cusp, $c_m$, between the strands labeled $k$ and $k+1$ at $x_{m-1}$.  Then, 
\[
d c_m = 1 + y^{m-1}_{k, k+1},
\]
and the matrix $\widetilde{Y}_{m-1}$ is most easily described entry by entry.  Let $\tau: \{ 1, \ldots, N(m) \} \rightarrow \{ 1, \ldots, N(m-1) \}$, be given by $\tau(i) = \left\{\begin{array}{lr} i, & i <k \\ i+2, & i \geq k\end{array}\right.$.  The $i, j$ entry of $\widetilde{Y}_{m-1}$ is given by 
\[
\widetilde{y}^{m-1}_{ij} = y^{m-1}_{\tau(i), \tau(j)} + a_{ij} \quad \mbox{where}
\]
$a_{ij} = y^{m-1}_{i, k+1} y^{m-1}_{k, \tau(j)} + y^{m-1}_{i, k} c_m y^{m-1}_{k, \tau(j)}  + y^{m-1}_{i, k+1} c_m y^{m-1}_{k+1, \tau(j)} + y^{m-1}_{i, k+1} (c_m)^2 y^{m-1}_{k+1, \tau(j)}$ when $i < k \leq j$ and $a_{ij} =0$ otherwise.

\begin{remark}  The technique of adding some variation of splashes to simplify the differential has been used in several places in the literature.  The version employed here is the same as that of \cite{FR} where we refer the reader for more details.  For an alternate approach, we expect that a DGA of the same form would arise from iterating the ``bordered Chekanov-Eliashberg algebra'' construction introduced in \cite{Si}.
\end{remark}

\subsection{Proof of Theorem \ref{the:Main}~(2)}

We begin by introducing a notation.  Given an involution $\tau$ of $\{1, \ldots, N\}$, $\tau^2 = \mathit{id}$, we let $B_\tau = (b_{ij})$ denote the $N \times N$ matrix with entries 
\[
b_{ij} = \left\{ \begin{array}{lc} 1, & \mbox{if } i < \tau(i) = j 
\\ 0, & \mbox{else} 
\end{array} \right..
\]

\subsubsection{{\bf (B)} $\Rightarrow$ {\bf (A)}}  Suppose that $L$ admits a generalized normal ruling, $\rho = (\rho_1, \ldots, \rho_m)$.  An augmentation $\varepsilon$ of the algebra $\mathcal{A}(L')$ is defined as follows:  on all right cusps, $c_m$, $\varepsilon(c_m) =0$; at crossings $b_m$, $\varepsilon(b_m)$ is $1$ if $b_m$ is a switch and $0$ otherwise; for all $m$, $\varepsilon(Y_m) = B_{\rho_m}$; and $\varepsilon(x^m_{i,j}) =0 $ for all $i,j$ except when a switch occurs between $x_{m-1}$ and $x_m$.  Assume the switch involves the $k$ and $k+1$ strands.  If one of the switching strands is also a fixed point strand, then of the generators $x^m_{ij}$ augment only $x^m_{k,k+1}$.  Else, note that due to the normality condition, near the switch the intervals connecting the switching strands and their companion strands (Remark \ref{rem:AltNR}) are either disjoint or nested.  Assume that the switch occurs between the strands labeled $k$ and $k+1$.  If the switch is disjoint, then augment only $x^m_{k,k+1}$.  If the switch is nested, then augment $x^m_{k,k+1}$ as well as $x^m_{\tau(k), \tau(k+1)}$ (resp. $x^m_{\tau(k+1), \tau(k)}$) if $\tau(k) < \tau(k+1)$ (resp. $\tau(k+1) < \tau(k)$).

It is straight forward to verify from the formulas of the previous section that $\varepsilon$ is an augmentation.  If $\rho$ is $p$-graded with respect to a Maslov potential $\mu$, then  $\varepsilon$ is as well.

\subsubsection{{\bf (A)} $\Rightarrow$ {\bf (B)}}
The proof of the reverse implication is based on some canonical form results from linear algebra due to Barannikov \cite{B}.  

\begin{definition}  An $M$-complex, $(V, \mathcal{B}, d)$ is a vector space $V$ over a field $\mathbb{F}$ with a chosen ordered basis $\mathcal{B} = \{v_1, \ldots, v_N\}$ together with a differential $d: V\rightarrow V$, $d^2=0$, of the form $\displaystyle d v_i = \sum_{i<j} c_{ij} v_j$.
\end{definition}

\begin{proposition} \label{prop:SpecBasis} If $(V, \mathcal{B}, d)$ is an $M$-complex, then there exists a triangular change of basis $\{\widetilde{v}_1, \ldots, \widetilde{v}_N \}$, $\displaystyle \widetilde{v}_i = \sum_{i\leq j} a_{ij} v_j$, and an involution $\tau: \{1, \ldots, N\} \rightarrow \{1, \ldots, N \}$ such that
\[
d \widetilde{v}_i= \left\{ \begin{array}{lc} \widetilde{v}_j, & \mbox{if } i < \tau(i) = j 
\\ 0, & \mbox{else} 
\end{array} \right..
\]
Moreover, the involution $\tau$ is unique.
\end{proposition}

\begin{remark}\label{rem:MatFunc}  

(i)  Suppose in addition that the basis elements $v_i$ are assigned degrees $|v_i| \in \Z/p$ so that $V$ is $\Z/p$-graded and $d$ has degree $-1$.  Then, the change of basis may be assumed to preserve degree.  Hence, if $i < \tau(i) = j$, then  $|v_i| = |v_j|+1$.  

(ii) The classes $[\widetilde{v}_i]$ such that $\tau(i) =i$ form a basis for the homology, $H(V,d)$.

(iii) Proposition \ref{prop:SpecBasis} has the following matrix interpretation:  There is a unique function, $D \mapsto \tau(D)$ which assigns to every  strictly upper-triangular $N\times N$ matrix, $D$, with $D^2 =0$ an involution $\tau= \tau(D)$ such that there exists an invertible upper-triangular matrix $P$ so that $PDP^{-1} = B_{\tau}$.  Notice that the uniqueness assertion implies that $\tau(QDQ^{-1}) = \tau(D)$ if $Q$ is non-singular and upper triangular.  
\end{remark}

\begin{proposition}[\cite{B}]  \label{prop:BLemma} Suppose that $(V, \mathcal{B}, d)$ is an $M$-complex, and $k \in \{1, \ldots, N\}$ is such that $d v_k = \sum_{k+1 < j} c_{kj} v_j$ so that $(V, \mathcal{B}', d)$ with $\mathcal{B}' = \{ v_1, \ldots, v_{k+1}, v_k, \ldots, v_N \}$ is also an $M$-complex.  
Then, the associated involutions $\tau$ and $\tau'$ are related as follows.  
\begin{enumerate}
\item It is always possible to have $\tau' = (k \,\, k+1) \circ \tau \circ (k \,\, k+1)$ where $(k \,\, k+1)$ denotes the transposition.

\item In the following cases, it is also possible to have $\tau' = \tau$:

\begin{enumerate}
\item If  $\tau(k+1) < \tau(k) < k < k+1$, $\tau(k) < k < k+1 < \tau(k+1)$, or $k < k+1 < \tau(k+1) < \tau(k)$.
\item If $\tau(k) < k < k+1 = \tau(k+1)$ or $\tau(k) = k < k+1 < \tau(k+1)$.
\item If $\tau(k) =k < k+1 = \tau(k+1)$.

\end{enumerate}

\end{enumerate}
\end{proposition}

\begin{remark} 

(i) From the matrix perspective, Proposition \ref{prop:BLemma} puts restrictions on $\tau(P_{k,k+1} D P_{k,k+1})$ when $P_{k,k+1}$ is the permutation matrix of the transposition $(k \,\, k+1)$ and the $k, k+1$-entry of $D$ is $0$.  

(ii)  Propositions \ref{prop:SpecBasis} and \ref{prop:BLemma} are essentially the same as Lemma 2 and Lemma 4 of \cite{B}.  Proposition \ref{prop:BLemma} is proven quite directly by considering cases.  

\end{remark}

\begin{proof}[Proof of {\bf (A)} $\Rightarrow$ {\bf (B)}]

Suppose now that $\varepsilon$ is an augmentation of $\mathcal{A}(L')$.  
 
For each $m$, the matrix $\varepsilon(Y_m)$ is strictly upper triangular and satisfies 
\[
[\varepsilon(Y_m)]^2 = \varepsilon \circ d (Y_m) = 0.
\]
Letting $\tau_m = \tau( \varepsilon(Y_m) )$ as in Remark \ref{rem:MatFunc} produces a sequence, $\tau_1, \ldots, \tau_M$, with $\tau_m$ an involution of  $\{1, \ldots, N(m)\}$.  We show that $\tau = (\tau_1, \ldots, \tau_M)$ satisfies the requirements of a generalized normal ruling.
 This requires establishing that the restrictions provided by Definitions \ref{def:NR} and \ref{def:GNR} on consecutive involutions, $\tau_{m-1}$ and $\tau_m$, are satisfied.  
 
Recall that each interval $(x_{m-1}, x_m)$ contains a single crossing or cusp.

\medskip

\noindent \emph{Case} $(x_{m-1}, x_m)$ contains a {\bf left cusp}:    Then equation (\ref{eq:diffe}) and the definition of augmentation allow us to compute
\begin{equation} \label{eq:conjugate} 
\varepsilon(Y_m) = (I + \varepsilon(X_m)) \varepsilon(\widetilde{Y}_{m-1}) (I + \varepsilon(X_m))^{-1}.
\end{equation}
Using Remark \ref{rem:MatFunc} we conclude that
\[
\tau_m = \tau( \varepsilon(Y_m) ) = \tau( \varepsilon(\widetilde{Y}_{m-1}) ).
\]
The $M$-complex associated with $\varepsilon(\widetilde{Y}_{m-1})$ is related to that of $\varepsilon(Y_{m-1})$ by adding two new generators, $v_k$ and $v_{k+1}$, to $\mathcal{B}$.  The complex is the split extension of that of $\varepsilon(Y_{m-1})$ by $\op{span}\{v_k, v_{k+1}\}$ with the differential $d v_k = v_{k+1}$.  It can then be checked directly from the definition that the involutions $\tau_{m-1}$ and $\tau_m$ satisfy (2) of Definition \ref{def:NR}.

\medskip

\noindent \emph{Case} $(x_{m-1}, x_m)$ contains a {\bf right cusp}:  Let $ \mathcal{C} =( V_{m-1}, \mathcal{B} = \{v_i \,|\, i =1, \ldots, N(m-1)\}, d)$ denote the $M$-complex associated with the matrix $\varepsilon(Y_{m-1})$ by the formula
\begin{equation} \label{eq:MCMatrix} \displaystyle
d v_i = \sum_{i < j} \varepsilon(y^{m-1}_{ij}) v_j.
\end{equation}
Note that $\tau_{m-1}$ is precisely the involution associated to $\mathcal{C}$ by Proposition \ref{prop:SpecBasis}.  From $0 = \varepsilon \circ d (c_m)$ we deduce that $1 = \varepsilon( y^{m-1}_{k,k+1})$, and it follows that $\tau_{m-1}(k) = k+1$.  

Next, one observes that $\varepsilon(\widetilde{Y}_{m-1})$ is the matrix of the $M$-complex 
\[
\widetilde{\mathcal{C}} = \left(\widetilde{V}_{m-1}, \widetilde{\mathcal{B}} = \{ [v_i] \, | \, i \neq k, k+1\}, \widetilde{d}\right)
\]
where $\widetilde{V}_{m-1}$ is the quotient of $V_{m-1}$ by the subcomplex $\{v_k + \varepsilon(c_m) v_{k+1}, d(v_k + \varepsilon(c_m) v_{k+1}) \}$ and $\widetilde{d}$ is the differential induced by $d$.  If $\{ \widetilde{v}_i \}$ is a triangular change of basis for $\mathcal{C}$ satisfying the conditions of Proposition \ref{prop:SpecBasis}, then $\{ [\widetilde{v}_i] \,|\, i \neq k, k+1 \}$ will be such a basis for $\widetilde{\mathcal{C}}$, so that the involution associated with $\varepsilon(\widetilde{Y}_{m-1})$ is related to $\tau_{m-1}$ as required in (iii) of Definition \ref{def:NR}.  Finally, using Equation (\ref{eq:conjugate}), $\tau_m = \tau(\varepsilon(Y_m)) = \tau(\varepsilon(\widetilde{Y}_{m-1}))$.

\medskip

\noindent \emph{Case} $(x_{m-1}, x_m)$ contains a {\bf crossing}, $b_m$:   We have 
\[
0 = \varepsilon \circ d (b_m) = \varepsilon(y^{m-1}_{k,k+1}).
\]
  Thus, $\varepsilon(\widehat{Y}_{m-1}) = \varepsilon(Y_{m-1})$ with both matrices having $0$ as their $k, k+1$ entry.  Then, compute that 

\smallskip

\leftline{$\varepsilon(B_{k,k+1}) \varepsilon(\widehat{Y}_{m-1}) \varepsilon(B_{k,k+1}^{-1}) =$}

\smallskip

\centerline{$ P_{k,k+1} [I + \varepsilon(b_m) E_{k, k+1}] \varepsilon(Y_{m-1}) [I + \varepsilon(b_m) E_{k, k+1}] P_{k,k+1}.$}

\smallskip

Regardless of the value of $\varepsilon(b_m)$, the $k, k+1$-entry of $$[I + \varepsilon(b_m) E_{k, k+1}] \varepsilon(Y_{m-1}) [I + \varepsilon(b_m) E_{k, k+1}]$$ is $0$, so the matrix $$A = \varepsilon(B_{k,k+1}) \varepsilon(\widehat{Y}_{m-1}) \varepsilon(B_{k,k+1}^{-1})$$ is strictly upper triangular and  $\tau(A)$ is related to $$ \tau\left((I + \varepsilon(b_m) E_{k, k+1}) \varepsilon(Y_{m-1}) (I + \varepsilon(b_m) E_{k, k+1})\right)= \tau(Y_{m-1}) = \tau_{m-1}$$ as in Proposition \ref{prop:BLemma}. It follows that 
\[
\tau_m  = \tau( \varepsilon(Y_m) ) = \tau( (I + \varepsilon(X_m)) A (I + \varepsilon(X_m))^{-1} ) = \tau(A)
\]
and $\tau_{m-1}$ satisfy the requirements near crossings (including the normality conditions) of Definition \ref{def:GNR}.

\medskip

The statement that $\tau$ is $p$-graded if $\varepsilon$ is $p$-graded follows from (i) of Remark \ref{rem:MatFunc}.  As in (\ref{eq:MCMatrix}), $\varepsilon(Y_m)$ is the matrix of an $M$-complex with basis $v_1, \ldots, v_{N(m)}$ corresponding to the strands of $L$ at $x_m$.  If $\varepsilon$ is $p$-graded with respect to $\mu$, then we can assign a grading by $|v_i| = \mu(m, i)$ and the differential will have degree $-1$.

\end{proof}

\end{document}